\documentclass[12pt]{amsart}

\usepackage{amsmath,amssymb,amsthm}
\usepackage{amsfonts}
\usepackage[mathscr]{eucal}
\usepackage{color}

\newtheorem{theorem}{Theorem}[section]
\newtheorem{lemma}[theorem]{Lemma}
\newtheorem{prop}[theorem]{Proposition}

\theoremstyle{definition}

\theoremstyle{remark}
\newtheorem{remark}[theorem]{Remark}

\numberwithin{equation}{section}

\newcommand{\rd}{{\mathbb R^d}}

\newcommand{\rr}{{\mathbb R}}

\def\Im {{\rm Im}\,}
\def\E{{\mathbb E}}

\begin{document}

\title{\bf Transient anomalous sub-diffusion on bounded domains}

\author{Mark M. Meerschaert}
\address{Mark M. Meerschaert, Department of Statistics and Probability,
Michigan State University, East Lansing, MI 48823.}
\email{mcubed@stt.msu.edu}
\urladdr{http://www.stt.msu.edu/$\sim$mcubed/}
\thanks{MMM was partially supported by NSF grants DMS-0125486, DMS-0803360, EAR-0823965 and NIH grant R01-EB012079-01.}

\author{Erkan Nane}
\address{Erkan Nane, Department of Mathematics and Statistics, 221 Parker Hall, Auburn University, Auburn, AL 36849.}
\email{nane@auburn.edu}
\urladdr{http://www.auburn.edu/$\sim$ezn0001/}

\author{P. Vellaisamy}
\address{P. Vellaisamy, Department of Mathematics,
 Indian Institute of Technology Bombay,
 Powai, Mumbai 400076 India.}
\email{pv@math.iit.ac.in}
\thanks{This paper was completed while PV was visiting Michigan State University. }

\begin{abstract}
This paper develops strong solutions and stochastic solutions for the tempered fractional diffusion equation on bounded domains.  First the eigenvalue problem for tempered fractional derivatives is solved.  Then a separation of variables, and eigenfunction expansions in time and space, are used to  write strong solutions.  Finally, stochastic solutions are written in terms of an inverse subordinator.
\end{abstract}

\keywords{Fractional diffusion, Cauchy problem, tempered stable, boundary value problem.}

\maketitle


\section{Introduction}

Transient anomalous sub-diffusion equations replace the first time derivative by a tempered fractional derivative of order $0<\beta<1$ to model delays between movements \cite{temperedLM, Cartea2007}.  These governing equations have proven useful in finance \cite{BNS, CGMY} and geophysics \cite{TemperedStable} to model phenomena that eventually transition to Gaussian behavior.   A stochastic model for transient anomalous sub-diffusion replaces the time variable in a diffusion by an independent inverse tempered stable subordinator.  This time-changed process is useful for particle tracking, a superior numerical method in the presence of irregular boundaries \cite{Zhang-JSP, Zhang-PRE2008}.  The idea of tempering was introduced by Mantegna and Stanley \cite{MantegnaStanley1994,MantegnaStanley1995} and developed further by Rosinski \cite{Rosinski2007}.

Section \ref{sec2} provides some background on diffusion and fractional calculus, to establish notation, and to make the paper relatively self-contained.  Section \ref{sec3} uses Laplace transforms and complex analysis to prove strong solutions to the eigenvalue problem for the tempered fractional derivative operator.  Then Section \ref{sec4} solves the tempered fractional diffusion equation on bounded domains.   Separation of variables and eigenvalue expansions in space and time lead to explicit strong solutions in series form.  Stochastic solutions are then developed, using an inverse tempered stable time change in the underlying diffusion process.

\section{Traditional and fractional diffusion}\label{sec2}
Suppose that $D$ is a bounded domain in $\rd$.  A uniformly elliptic operator in divergence form is defined for $u\in C^2(D)$ by
\begin{equation}\label{unif-elliptic-op}
L_Du=\sum_{i,j=1}^{d} \frac{\partial }{\partial x_j} \left(a_{ij}(x)\frac{\partial
u}{\partial x_i}\right)
\end{equation}
with $a_{ij}(x)=a_{ji}(x)$ such that for some $\lambda>0 $ we have
\begin{equation}\label{elliptic-bounds}
\lambda \sum_{i=1}^ny_i^2\leq \sum_{i,j=1}^na_{ij}(x)y_iy_j\leq
\lambda^{-1} \sum_{i=1}^ny_i^2,\quad\text{for all $ y \in \rd$.}
\end{equation}
We also assume that for some $\Lambda>0$ we have
\begin{equation}\label{aijbound}
\sum_{i,j=1}^n |a_{ij}(x)| \leq \Lambda\quad\text{for all $x \in D$.}
\end{equation}
Take $a=\sigma \sigma ^T$, and $B(t)$ a Brownian motion.  Let $X(t)$ solve the stochastic differential equation $dX(t)=b(X(t))dt+\sigma (X(t))dB(t)$, and define the first exit time  $\tau_D(X)=\inf \{ t\geq 0:\ {X(t)}\notin D\}$.  An application of the It$\mathrm{\hat{ o}}$ formula shows that the semigroup
 \begin{equation}\label{eqn2.3n}
T_{D}(t)f(x)={\E_x[f(X(t))I(\tau_D(X)>t)]}
\end{equation}
has generator \eqref{unif-elliptic-op}, see Bass \cite[Chapters 1 and 5]{bass}.
Since $T_D(t)$ is intrinsically ultracontractive (see
 \cite[Corollary 3.2.8, Theorems 2.1.4, 2.3.6, 4.2.4  and Note 4.6.10]{davies} and  \cite[Theorems 8.37 and 8.38]{gilbarg-trudinger}),
there exist eigenvalues
$0< \eta_1<\eta_2\leq \eta_3\cdots$, with $\eta_n\to\infty$ and a
complete orthonormal basis of eigenfunctions
$\psi_n$ in $L^2(D)$ satisfying
\begin{equation}\label{eigen-eigen}
 L_D \psi_n(x)=-\eta_n \psi_n(x), \ x\in D:\  \\
\psi_n |_{\partial D}=0.
\end{equation}
Then $p_D(t,x,y)=\sum_{n=1}^{\infty}e^{-\eta_n t}\psi_n(x)\psi_n(y)$
is the heat kernel of the killed semigroup $T_D$. This series
converges absolutely and  uniformly on $[t_0,\infty)\times D\times
D$ for all $t_0>0$.

Denote the Laplace transform (LT) $ t\rightarrow s $ of $u(t,x)$
by
\begin{equation*}\begin{split}
\tilde{u}(s,x)&= \mathcal{L}_t [u(t, x)]=\int_{0}^{\infty}e^{-s
t}u(t,x)dt .
\end{split}\end{equation*}
The $\psi_n$-transform is defined by
$\bar{u}(t,n)=\int_{D}\psi_n(x)u(t,x)dx$
and the $\psi_n$-Laplace transform is defined by
 \begin{equation}\begin{split} \label{eqn2.5v3}
\hat{u}(s,n)
&= \int_{D} \psi_{n}(x) \tilde{u}(s, x) dx .
 \end{split}\end{equation}
Since $\{\psi_n\}$ is a complete orthonormal basis for $L^2(D)$,
we can invert the $\psi_n$-transform to obtain
$u(t,x)=\sum_n \bar{u}(t,n) \psi_n(x)$
for any $t>0$, where the series converges
in the $L^2$ sense \cite[Proposition 10.8.27]{Royden}.

Suppose that $D$ satisfies a uniform exterior cone condition, so that
then each $x\in \partial D$ is regular for $D^\complement$ \cite[Proposition 1, p.\ 89]{bass2}.
If $f$ is continuous on $\bar D$, then
\begin{equation}\begin{split}\label{TDdef}
u(t,x)&=T_D(t)f(x)={\E_x[f(X(t))I(\tau_D(X)>t)]}\\
&=\int_D p_D(t,x,y)f(y)dy=\sum_{n=1}^{\infty}e^{-\eta_n t}\psi_n(x)\bar{f}(n)
\end{split}\end{equation}
solves the Dirichlet initial-boundary value problem \cite[Theorem
2.1.4]{davies}:
\begin{equation}\begin{split}\label{CP}
\frac{\partial u(t,x)}{\partial t}&= L_D u(t,x),\ \ x\in D, \ t>0, \\
u(t,x)&=0,\ \ x\in \partial D,\\
u(0,x)&=f(x), \ \ x\in D.
\end{split}\end{equation}
This shows that the diffusion $X(t)$ killed at the boundary $\partial D$ is the stochastic solution to the diffusion equation \eqref{CP} on the bounded domain $D$.

The Riemann-Liouville
fractional derivative \cite{MillerRoss, Samko} is defined by
\begin{equation}\label{RLDef}
{\frac{\partial^\beta}{\partial t^\beta}}\,
g(t)=\frac{1}{\Gamma(1-\beta)}\frac d{dt}\int_0^t
\frac{g(s)\,ds}{(t-s)^\beta}
\end{equation}
when $0<\beta<1$.  The Caputo fractional derivative \cite{Caputo} is defined by
\begin{equation}\label{CaputoDef}
\left(\frac{\partial}{\partial t}\right)^\beta g(t)=\frac{1}{\Gamma(1-\beta)}\int_0^t  \frac{g'(s)\,ds}{(t-s)^\beta}
\end{equation}
when $0<\beta<1$.  It is easy to check using ${\mathcal{L}}[t^{-\beta}]=s^{\beta-1}\Gamma(1-\beta)$
that
\begin{equation}\label{rld1}
 {\mathcal{L}}_t \left[\frac{d^{\beta}}{dt^{\beta}} g(t)\right] = s^{\beta}
 \tilde g(s),
\end{equation}
while the Caputo fractional derivative \eqref{CaputoDef} has Laplace
 transform $s^\beta\tilde g(s)-s^{\beta-1}g(0)$.  It follows that
\begin{equation}\label{RLtoCaputo}
{\frac{\partial^\beta}{\partial t^\beta}}\, g(t)=
\left(\frac{\partial}{\partial t}\right)^\beta g(t) +
\frac{g(0)t^{-\beta}}{\Gamma(1-\beta)}.
\end{equation}
Substituting a Caputo fractional derivative of order $0<\beta<1$ for the first order time derivative in \eqref{CP} yields a fractional Cauchy problem. This fractional diffusion equation was solved in \cite[Theorem 3.6]{m-n-v-aop}.  Those solutions exhibit anomalous sub-diffusion at all times, with a plume spreading rate that is significantly slower than \eqref{TDdef}.  Many practical problems exhibit transient sub-diffusion, converging to \eqref{TDdef} at late time \cite{BNS, CGMY, TemperedStable}.  Hence, the goal of this paper is to extend the results of \cite{m-n-v-aop} to transient sub-diffusions.

\section{Eigenvalues for tempered fractional derivatives}\label{sec3}

Suppose $D(x)$ is a standard stable subordinator with L\'evy measure $\phi(y,\infty)=I(y>0)y^{-\beta}/\Gamma(1-\beta)$  for $0<\beta<1$ , so that ${\mathbb E}[e^{-sD(x)}]=e^{-x \psi(s)}$, where the Laplace symbol
$\psi(s)=s^\beta=\int_0^\infty \left(1-e^{-sy}\right)\phi_\beta(dy) .$
If $f_x(t)$ is the density of $D(x)$, then $q_{\lambda}(t,x)=
f_x(t)e^{-\lambda t}/e^{-x \lambda^\beta}$ is a density on $x>0$ with Laplace transform (LT)
\begin{equation}\label{temperedLT}
\tilde q_{\lambda}(s,x)=\int_0^\infty q_{\lambda}(t,x)\,dt=e^{x \lambda^\beta}\int_0^\infty e^{-(s+\lambda)t} f_x(t)\,dt=
  e^{-x\psi_{\lambda}(s)}
\end{equation}
where $\psi_{\lambda}(s) =(s+\lambda)^\beta-\lambda^\beta$.  Rosinski \cite{Rosinski2007} notes that the tempered stable subordinator $D_{\lambda}(x)$ with this Laplace symbol has L\'evy measure $\phi_\lambda(dy)=e^{-\lambda y}\phi_\beta(dy)$.

Define the Riemann-Liouville tempered fractional derivative of order $0<\beta<1$ by
\begin{equation}\label{RLtfd}
{\frac{\partial^{\beta,\lambda}}{\partial t^{\beta,\lambda}}}\,g(t)=e^{-\lambda t} \frac{1}{\Gamma(1-\beta)}\frac d{dt}\int_0^t
\frac{e^{\lambda s}g(s)\,ds}{(t-s)^\beta}-\lambda^\beta g(t)
\end{equation}
as in \cite{temperedLM}.  We say that a function is a mild solution to a
pseudo-differential equation if its Laplace transform solves the
corresponding equation in transform space.

\begin{prop}
The density $q_{\lambda}(t,x) $ of the tempered stable subordinator $D_{\lambda}(x)$ is a mild solution to
{\begin{equation}\label{gov}
\frac \partial{\partial x}\, q_{\lambda}(t,x) = - {\frac{\partial^{\beta,\lambda}}{\partial t^{\beta,\lambda}}}\,  q_{\lambda}(t,x) .
\end{equation}}
\end{prop}

\begin{proof}
Clearly  $\tilde q_{\lambda}(s,x)=e^{-x\psi_{\lambda}(s)}$ solves
\begin{equation}\label{govLT}
\frac {\partial }{\partial x}\, \tilde q_{\lambda}(s,x) =
-\psi_{\lambda}(s)\, \tilde q_{\lambda}(s,x)
\end{equation}
with initial condition $\tilde q_{\lambda}(s,0)=1$.  The right-hand side of \eqref{govLT} involves a pseudo-differential operator $\psi_{\lambda}(\partial_t)$ with Laplace symbol $\psi_{\lambda}(s)$, see Jacob \cite{Jacob}.  To complete the proof, it suffices to show that $\psi_{\lambda}(s)\tilde g(s)$ is the LT of \eqref{RLtfd}.    Since ${\mathcal{L}}[e^{\lambda t} g(t)] = \tilde g(s-\lambda)$,
we get
\begin{equation}\label{rld2}
 {\mathcal{L}} \left[\frac{d^{\beta}}{dt^{\beta}} \left(e^{\lambda t}g(t)\right)\right] = s^{\beta}
 \tilde g(s-\lambda),
\end{equation}
which leads to
\begin{equation}\label{rld3}
 {\mathcal{L}} \left[e^{-\lambda t} \frac{d^{\beta}}{dt^{\beta}} \big(e^{\lambda t}g(t)\big)\right] = (s+\lambda)^{\beta}
 \tilde g(s).
\end{equation}
Then
\eqref{gov} follows easily.  This also shows that $\psi_{\lambda}(\partial_t)$ is the negative generator of the $C_0$ semigroup associated with the tempered stable process.
\end{proof}

Define the inverse tempered stable subordinator
\begin{equation}\label{ElambdaDef}
E_{\lambda}(t) = \inf \{ x>0  : D_{\lambda} (x) >t\}.
\end{equation}
A general
result on hitting times \cite[Theorem 3.1]{M-S-triangular} shows
that, for all $t>0$, the random variable $E_{\lambda}(t)$ has
Lebesgue density
\begin{equation}\label{Edensity}
g_{\lambda}(t,x)=\int_0^t\phi_\lambda(t-y,\infty)q_{\lambda}(y,
x)\,dy
\end{equation}
and $(t,x)\mapsto {g_{\lambda}(t,x)}$ is
measurable.  Following \cite[Remark 4.8]{M-S-triangular}, we define the Caputo tempered fractional derivative of order $0<\beta<1$ by
\begin{equation}\label{tcd-1}
\left(\frac {\partial}{\partial t}\right)^{\beta,\lambda} g(t)
={\frac{\partial^{\beta,\lambda}}{\partial t^{\beta,\lambda}}}\,g(t)
-\frac{g(0)}{\Gamma(1-\beta)} \int_t^\infty
e^{-\lambda r} \beta r^{-\beta-1}\,dr .
\end{equation}

\begin{prop}
The density \eqref{Edensity} of the inverse tempered stable subordinator \eqref{ElambdaDef} is a mild solution to
\begin{equation}\label{gCpC}
\frac\partial{\partial x}\, g_\lambda(t,x)=-\left(\frac {\partial}{\partial t}\right)^{\beta,\lambda} g_\lambda(t,x) .
\end{equation}
\end{prop}

\begin{proof}
Theorem 4.1 in \cite{M-S-triangular} shows that \eqref{Edensity} is a mild solution to the pseudo-differential equation
\begin{equation}\label{gCp}
\frac\partial{\partial x}\, g_\lambda(t,x)=-\psi_{\lambda}(\partial_t)g_\lambda(t,x)+\delta(x)\phi_\lambda(t,\infty)
\end{equation}
where $\psi_{\lambda}(\partial_t)$ is the pseudo-differential
operator with Laplace symbol $\psi_{\lambda}(s)$, i.e.,
the Riemann-Liouville tempered fractional derivative \eqref{RLtfd}. Use
\begin{equation}\label{LMtail}
\phi_\lambda(t,\infty)=\frac{1}{\Gamma(1-\beta)} \int_t^\infty
e^{-\lambda r} \beta r^{-\beta-1}\,dr
\end{equation}
to rewrite \eqref{gCp} in the form
\begin{equation}\label{gCpRL}
\frac\partial{\partial x}\, g_\lambda(t,x)=-{\frac{\partial^{\beta,\lambda}}{\partial t^{\beta,\lambda}}}\, g_\lambda(t,x)+\frac{\delta(x)}{\Gamma(1-\beta)} \int_t^\infty
e^{-\lambda r} \beta r^{-\beta-1}\,dr ,
\end{equation}
and then apply \eqref{tcd-1} with $g_\lambda(0,x)=\delta(x)$ to get \eqref{gCpC}.
\end{proof}

The next two results establish eigenvalues for Caputo tempered fractional derivatives, which will then be used in Section \ref{sec4} to solve tempered fractional diffusion equations by an eigenvalue expansion.


\begin{lemma}\label{eigenvalue-problem}
For any $\mu>0$, the Laplace transform
\begin{equation}\label{gLTxtomu}
\check g_\lambda (t, \mu)={\mathcal{L}_x [g_\lambda (t,x)]}= \int_0^\infty e^{-\mu
x}g_\lambda(t,x)\,dx =\E [e^{-\mu
E_{\lambda}(t)}]
 \end{equation}
is a mild solution to the eigenvalue problem
{\begin{equation}\label{laplace-density-pde}
\left(  \frac {\partial}{\partial t}
\right)^{\beta,\lambda}\check{g_\lambda}(t, \mu)=-\mu
\check{g_\lambda}(t, \mu)
\end{equation}}
with $\check{g_\lambda}(0, \mu)=1$, for the Caputo tempered fractional derivative \eqref{tcd-1}.
\end{lemma}

\begin{proof}
Equation (3.12) in \cite{M-S-triangular} shows that
\begin{equation}\label{LTlevytail}
\mathcal{L}_t[\phi_\lambda (t, \infty)]=s^{-1}\psi_\lambda(s) .
\end{equation}
Then \eqref{Edensity} together with the LT convolution property shows that
\begin{equation}\label{laplace-distributed-inverse}
\tilde g_\lambda(s,x)=
\mathcal{L}_t[g_\lambda(t,x)] = \mathcal{L}_t[\phi_\lambda (t, \infty)] \mathcal{L}_t
[q_\lambda(t,x)]=\frac{1}{s}\psi_\lambda(s)e^{-x\psi_\lambda(s)}
\end{equation}
for any $x>0$, and then a Fubini argument shows that the double Laplace transform
\begin{equation}\label{double-laplace}
\begin{split}
G_\lambda(s,\mu)= \mathcal {L}_t\mathcal {L}_x[g_\lambda(t,x)]
&= \frac{\psi_\lambda(s)}{s} \int_0^\infty e^{-(\mu + \psi_\lambda(s))x}dx = \frac{\psi_\lambda(s)}{s(\mu +\psi_\lambda(s))} .
\end{split}
\end{equation}
Rearrange \eqref{double-laplace} to get
\begin{equation}\label{eqn3.16n}
-\mu G_\lambda(s,\mu)=\psi_\lambda(s)G_\lambda(s,\mu)- s^{-1}\psi_\lambda(s) ,
\end{equation}
use \eqref{LTlevytail} along with \eqref{tcd-1} and \eqref{LMtail} to see that
\begin{equation}\label{tcd-2}
\mathcal{L}_t \left[\left(\frac {\partial}{\partial t}
\right)^{\beta,\lambda} \check{g_\lambda}(s, \mu)\right]
= \psi_\lambda(s) G_\lambda(s, \mu)- s^{-1}\psi_\lambda(s) ,
 \end{equation}
then substitute into \eqref{eqn3.16n} to get
\[
\mathcal{L}_t \left[\left(  \frac {\partial}{\partial t}
\right)^{\beta,\lambda}\check{g_\lambda}(t, \mu)\right]
=\mathcal{L}_t \left[-\mu
\check{g_\lambda}(t, \mu)\right] .
\]
This proves that \eqref{gLTxtomu} is the mild solution to \eqref{laplace-density-pde}.
\end{proof}

The next theorem is the main technical result of this paper.  It shows that Laplace transforms \eqref{gLTxtomu} of inverse tempered stable densities
are the eigenvalues of the Caputo tempered fractional derivative, in the strong sense.

\begin{theorem}\label{classical--tempered-caputo}
{For $0<\beta<1$, let }
\begin{equation}\label{kdef}
k(t)=\frac{ e^{-t\lambda}}{\pi \sin
(\beta\pi )}t^{\beta-1}\Gamma(1-\beta) .
\end{equation}
For any $\mu,\lambda>0$, $\mu \neq \lambda^\beta $, the function $\check{g_\lambda}(t,
\mu)$ in \eqref{gLTxtomu}
can be written in the form
\begin{equation}\label{inverse-laplace}
\check{g_\lambda}(t,\mu)=\frac{\mu}{\pi}\int_0^\infty
(r+\lambda)^{-1}e^{-t(r+\lambda)}\Phi(r,1)dr,
\end{equation}
where
$$\Phi(r,1)=\frac{r^\beta \sin (\beta\pi )}{r^{2\beta}\sin^2 (\beta \pi)+(\mu-\lambda^\beta+r^\beta
\cos (\beta\pi))^2}$$
and
\begin{equation}\label{time-derivative-bound}
|\partial_t \check{g_\lambda}(t,\mu)|\leq \mu k(t).
\end{equation}
Then $\check{g_\lambda}(t,\mu)$ is a strong (classical) solution
of the eigenvalue problem \eqref{laplace-density-pde}.
\end{theorem}

\begin{proof}
The proof extends Theorem 2.3 in Kochubei \cite{koch3} using some
probabilistic arguments.  Since  $E_ {\lambda}(t)$ has
continuous sample paths, a dominated convergence argument shows
that $\check{g_\lambda}(t, \mu)=\E [e^{-\mu E_{\lambda}(t)}]$ is a
continuous function of $t>0$. Use \eqref{double-laplace} to write
\begin{equation}\label{Laplace-laplace-density}
G_\lambda(s,\mu)= \mathcal
{L}_t[\check{g_\lambda}(t,\mu)]=\frac{\psi_\lambda(s)}{s(\mu
+\psi_\lambda(s))}=\frac{[(s+\lambda)^\beta
-\lambda^\beta]}{s[(s+\lambda)^\beta-\lambda^\beta+\mu]}
\end{equation}
and note that $G_\lambda(s,\mu)$ is analytic off the branch cut
$\arg(s)=\pi, |s|\geq 0$.   The Laplace inversion formula
\cite[p.\ 25]{ditkin} shows that for suitable $\gamma>0$ and for
almost all $t>0$,
\begin{equation}\label{inverse-derivative-laplace}
\check{g_\lambda}(t,\mu)=\frac{d}{dt}\frac{1}{2\pi
i}\int_{\gamma-i\infty}^{\gamma +i\infty}\frac{e^{st}}{s}
\frac{[(s+\lambda)^\beta-\lambda^\beta]/s}{[(s+\lambda)^\beta-\lambda^\beta+\mu]}ds.
\end{equation}
Let $\frac12<\omega<1$ and consider the closed curve
$C_{\gamma,\omega}$ in $\mathbb{C}$, formed by a circle of radius
$R_n$ with a counterclockwise orientation, cut off on the right
side by the line $\Im(z)=\gamma$, and by the curve $S_{\gamma,
\omega}$ on the left side, consisting of the arc
$$
T_{\gamma,\omega}=\{s\in \mathbb{C}: \ |s|=\gamma, \ \ |\arg s|\leq\omega \pi\}
$$
and the two rays
\begin{equation*}
\begin{split}
\Gamma^+_{\gamma,\omega}&=\{s\in \mathbb{C}: \ |s|\geq \gamma, \ \ |\arg s|=\omega \pi\} ,\\[3pt]
\Gamma^-_{\gamma,\omega}&=\{s\in \mathbb{C}: \ |s|\geq \gamma, \ \ |\arg s|=-\omega \pi\}.
\end{split}
\end{equation*}
By Cauchy's Theorem, the integral
\[\int_{C_{\gamma,\omega}}\frac{e^{st}}{s}\frac{[(s+\lambda)^\beta-\lambda^\beta]/s}{[(s+\lambda)^\beta-\lambda^\beta+\mu]}ds=0\]
and then Jordan's Lemma \cite[p.\ 27]{ditkin} implies that we can let $R_n\to\infty$ to get
\begin{equation}\label{inverse-contour-laplace}
\check{g_\lambda}(t,\mu)=\frac{d}{dt}\frac{1}{2\pi
i}\int_{S_{\gamma,\omega}}\frac{e^{st}}{s}\frac{[(s+\lambda)^\beta-
\lambda^\beta]/s}{[(s+\lambda)^\beta-\lambda^\beta+\mu]}ds.
\end{equation}
%
%
%
%
%
%
Now pass the derivative inside the integral to get
\begin{equation}\label{inverse-contour-laplace-1}
\check{g_\lambda}(t,\mu)=\frac{1}{2\pi
i}\int_{S_{\gamma,\omega}}e^{st}\frac{[(s+\lambda)^\beta-\lambda^\beta]/s}
{[(s+\lambda)^\beta-\lambda^\beta+\mu]}ds
\end{equation}
which also implies the smoothness of the function
$t\to\check{g_\lambda}(t,\mu)$.  It is not hard to check that the integral over $T_{\gamma, \omega}$ tends to zero as $\gamma\to 0$.  Compute the remaining path integral, and let $\omega\to 1$ to obtain \eqref{inverse-laplace}.

Differentiate \eqref{inverse-laplace} with respect to $t$ and use $r^{\beta} \sin
(\beta \pi) \Phi(r, 1) \leq 1 $ to write
\begin{equation} \label{tdb1}
\begin{split}
|\partial_t
\check{g_\lambda}(t,\mu)|&=\left|\frac{\mu}{\pi}\int_0^\infty
(r+\lambda)^{-1}[\partial_te^{-t(r+\lambda)}]\Phi(r,1)dr   \right|\\
&\leq \frac{\mu}{\pi \sin (\beta\pi )}\int_0^\infty e^{-t(r+\lambda)}r^{-\beta}dr\\
&=\frac{\mu e^{-t\lambda}}{\pi \sin (\beta\pi )}t^{\beta-1}\Gamma(1-\beta)= \mu k(t),
\end{split}
\end{equation}
so that \eqref{time-derivative-bound} holds.
 Note that $|\check{g_\lambda}(t, \mu)|\leq 1$,  and write
\begin{equation*}
\begin{split}
 \bigg|\left(\frac {\partial }{\partial t}\right)^{\beta}\left(e^{\lambda t}\check{g_\lambda}(t,\mu)\right)\bigg|
&=\bigg| \frac{1}{\Gamma(1-\beta)}\int_0^t \bigg( \lambda e^{\lambda s} \check{g_\lambda}(s,\mu)+e^{\lambda s}\frac{\partial[ \check{g_\lambda}(s,\mu)]}{\partial s}\bigg)\frac{ds}{(t-s)^\beta} \bigg|\\
&\leq\frac{1}{\Gamma(1-\beta)}\int_0^t \bigg( \lambda e^{\lambda s} |\check{g_\lambda}(s,\mu)|+e^{\lambda s}\bigg|\frac{\partial[\check{g_\lambda}(s,\mu)]}{\partial s}\bigg|\bigg)\frac{ds}{(t-s)^\beta} \\
&=\frac{1}{\Gamma(1-\beta)} \int_0^t \bigg( \lambda
e^{\lambda s} +\frac{\mu }{\pi \sin (\beta\pi
)}s^{\beta-1}\Gamma(1-\beta) \bigg)\frac{ds}{(t-s)^\beta}  .
\end{split}
\end{equation*}
Then a simple dominated convergence argument shows that the Riemann-Liouville fractional derivative of
$e^{\lambda t}\check{g_\lambda}(t,\mu)$ is a continuous function of $t>0$.
Now it follows from \eqref{RLtfd} and \eqref{tcd-1} that the Caputo tempered fractional derivative of $\check{g_\lambda}(t,\mu)$ is continuous in $t>0$. Since both sides of
\eqref{laplace-density-pde} are continuous in $t>0$, it follows from Lemma \ref{eigenvalue-problem} and the
uniqueness theorem for the Laplace transform that \eqref{laplace-density-pde} holds pointwise in $t>0$ for all $\mu>0$.
\end{proof}

\section{Tempered fractional diffusion}\label{sec4}

Let $D_\infty=(0,\infty )\times D$, define $\mathcal{H}_{L_D}(D_\infty)= \{u:D_\infty\to \rr :\ \ L_D
u(t,x)\in C(D_\infty)\}$, and let $\mathcal{H}_{L_D}^{b}(D_\infty)= \mathcal{H}_{L_D}(D_\infty) \cap
\{u: |\partial_t u(t,x)|\leq k(t)g(x),\ \  g\in L^\infty(D), \ t>0\}$,
where $k(t)$ is defined in \eqref{kdef}.

\begin{theorem}\label{FC-PDE-Lx}
Let $D$ be a bounded domain with $\partial D \in  C^{1,\alpha}$ for some $0<\alpha<1$, and let $X(t)$ be a continuous Markov process with generator \eqref{unif-elliptic-op}, where $a_{ij}\in C^{\alpha}(\bar D)$.
Then, for any $f\in D(L_D)\cap C^1(\bar D)\cap C^2( D) $  such that the eigenfunction expansion of $L_{{D}}f$ with respect to the complete orthonormal basis $\{\psi_n\}$ converges uniformly and absolutely,
the (classical) solution  of
\begin{eqnarray}
\left(\frac{\partial} {\partial t}\right)^{\beta,\lambda}u(t,x)
&=&
L_D u(t,x),  \  \ x\in D, \ t\geq 0\label{frac-derivative-bounded-dd};\\[-5pt]
u(t,x)&=&0, \ x\in \partial D,\ t\geq 0; \nonumber\\
u(0,x)& =& f(x), \ x\in D,\nonumber
\end{eqnarray}
for  $u  \in
\mathcal{H}_{L_D}^{b}(D_\infty)\cap C_b(\bar D_\infty) \cap C^1(\bar D)$,
 is given by
\begin{eqnarray}
u(t,x)&=&{\E_{x}}[{f(X(E_{\lambda}{(t)}))I( \tau_D(X)> E_{\lambda}(t))}]={\E_{x}}
[{f(X(E_{\lambda}{(t)}))I( \tau_D(X(E_{\lambda}))> t)}]\nonumber\\
&=& \int_{0}^{\infty}T_{D}(l)f(x)g_\lambda(t,l)dl= \sum_0^\infty
\bar{f}(n)\psi_n(x)\check{g_\lambda}(t,
\eta_n),\label{stoch-rep-L}
\end{eqnarray}
where $E_{\lambda}(t)$ is defined by \eqref{ElambdaDef}, independent of $X(t)$, $\check{g_\lambda}(t,{\eta})=\E (e^{-\eta
E_{\lambda}(t)})$ is its Laplace transform, and $T_D(t)$ is the killed semigroup \eqref{eqn2.3n}.
\end{theorem}

\begin{proof}
The proof uses a separation of variables.  Suppose $u(t,x)=G(t)F(x)$ is a solution of
(\ref{frac-derivative-bounded-dd}),
substitute into (\ref{frac-derivative-bounded-dd}) to get
$$
F(x)\left(\frac d{dt}\right)^{\beta,\lambda} G(t)
= G(t)L_D F(x)
$$
and divide both sides by $G(t)F(x)$ to obtain
$$
\frac{\left(\frac d{dt}\right)^{\beta,\lambda} G(t)}{G(t)} = \frac{L_D F(x)}{F(x)}=-\eta.
$$
Then we have
\begin{equation}\label{time-pde}
\left(\frac d{dt}\right)^{\beta,\lambda} G(t)=-\eta G(t), \ t>0
\end{equation}
and
\begin{equation}\label{space-pde}
LF(x)=-\eta F(x), \ x\in D, \ F|_{\partial D}=0.
\end{equation}
The eigenvalue problem (\ref{space-pde}) is solved by an infinite
sequence of pairs $\{(\mu_n, \psi_n)\}$, where $0<
\eta_1<\eta_2\leq \eta_3 \leq \cdots$, $\eta_n\to\infty$, as $n\to\infty$, and $\psi_n$ forms a complete orthonormal set in $L^2(D)$.
In particular, the initial function $f$ regarded as an element of $L^2(D)$ can be represented as
\begin{equation}
f(x)=\sum_{n=1}^\infty \bar f(n)\psi_n(x).
\end{equation}
Use Lemma \ref{eigenvalue-problem} to see that $G_n(t)=\bar f(n)\check{g_\lambda}(t, \eta_n)$
solves \eqref{time-pde}.  Sum these solutions $\psi_n(x)G_n(t)$ to (\ref{frac-derivative-bounded-dd}), to get
\begin{equation}\label{formal-sol-L-1}
u(t,x)=\sum_{n=1}^{\infty}\bar{f}(n)\check{g_\lambda}(t,
\eta_n)\psi_n(x) .
\end{equation}
It remains to show that \eqref{formal-sol-L-1} solves \eqref{frac-derivative-bounded-dd} and satisfies the conditions of Theorem \ref{FC-PDE-Lx}.

The remainder of the proof is similar to \cite[Theorem 3.1]{m-n-v-aop}, so we only sketch the argument.  First note that \eqref{formal-sol-L-1} {converges} uniformly in $t\in [0,\infty)$ in the $L^2$ sense.  Next argue $||u(t, \cdot)-f||_{2,D}\to 0$ as $t\to 0$ using the fact that, since $\check{g_\lambda}(t, \lambda)$  is the Laplace transform of
{ $E_\lambda(t)$,} it is completely monotone and
non-increasing in $\lambda\geq 0$.  Use the Parseval identity, the fact that $\mu_n$ is increasing in $n$, and the
fact that $\check{g_\lambda}(t,\mu_n)$ is non-increasing {in} $n\geq 1$, to get
$
||u(t,\cdot)||_{2,D} \leq \check{g_\lambda}(t,\mu_1)||f||_{2,D}.
$
A Fubini argument, which can be rigorously justified using the bound $\left|\partial_t u(t,x)\right|\leq k(t)g(x)$ from Theorem \ref{classical--tempered-caputo}, shows that the $\psi_n$ transform commutes with
the Caputo tempered fractional derivative. For this, it suffices to show that the
$\psi_n$-transform commutes with the Caputo fractional derivative
of $e^{\lambda t}u(t,x)$.  To check this, write
\begin{equation*}
\begin{split}
&\int_D \psi_n(x) \left(\frac {\partial}{\partial t}\right)^{\beta}\left(e^{\lambda t}u(t,x)\right) dx\\
&=\int_D \psi_n(x) \frac{1}{\Gamma(1-\beta)}\int_0^t \frac{\partial \left( e^{\lambda s}u(s,x)\right)}{\partial s}\frac{ds}{(t-s)^\beta} dx\\
&= \frac{1}{\Gamma(1-\beta)}\int_0^t \left(\int_D \psi_n(x) \frac{\partial }{\partial s}\left(e^{\lambda s}u(s,x)\right)dx\right) \frac{ds}{(t-s)^\beta} \\
&= \frac{1}{\Gamma(1-\beta)}\int_0^t \frac{\partial }{\partial s}\left(e^{\lambda s}\int_D \psi_n(x) u(s,x)dx\right) \frac{ds}{(t-s)^\beta}  \\
&=\frac{1}{\Gamma(1-\beta)}\int_0^t \frac{\partial }{\partial s}(e^{\lambda s}\bar u(s,n)) \frac{ds}{(t-s)^\beta} =\left(\frac {\partial}{\partial
t}\right)^{\beta}\left(e^{\lambda t} \bar u(t,n)\right).
 \end{split}\end{equation*}
Then the Caputo tempered fractional time derivative and the generator $L_D$ can be
applied term by term in \eqref{formal-sol-L-1}.  Next show that the series \eqref{formal-sol-L-1} is the classical
solution to \eqref{frac-derivative-bounded-dd} by checking uniform and absolute convergence.   Argue that $u\in
C^1(\bar D)$ using \cite[Theorem 8.33]{gilbarg-trudinger}, and the absolute and uniform convergence of the series defining
 $f$.   Finally, obtain the stochastic solution by inverting  the $\psi_n$-Laplace transform.
Since $\{\psi_n\}$ forms a complete orthonormal basis for
$L^2(D),$ the $\psi_n$-transform of the killed semigroup
$T_{{D}}(t)f(x)=\sum_{m=1}^\infty e^{-\mu_m t}\psi_m
(x)\bar f (m)$ from \eqref{TDdef} is given by
\begin{equation}\label{MMMc}
\begin{split}
\overline{[T_{{D}}(t)f]}(n)
&=e^{-t\mu_n}\bar{f}(n).
 \end{split}
 \end{equation}
Use Fubini together with \eqref{formal-sol-L-1}
and  \eqref{MMMc} to get
\begin{eqnarray}
u(t,x)&=&\sum_{n=1}^{\infty}\bar{f}(n)\psi_n(x)\check{g_\lambda}(t,\mu_n)=
\sum_{n=1}^{\infty}\psi_n(x) \int_0^{\infty}\bar{f}(n)e^{-\mu_n y}g_\lambda(t,y)dy\nonumber\\
&=&\sum_{n=1}^{\infty}\psi_n(x)\int_0^{\infty}\overline{[T_{{D}}(y)f]}(n)g_\lambda(t,y)dy\nonumber\\
&=&\int_0^{\infty}\left[\sum_{n=1}^{\infty}\psi_n(x)\bar{f}(n)e^{-y\mu_n}\right]g_\lambda(t,y)dy\nonumber\\
&=&\int_{0}^{\infty}T_{{D}}(y)f(x)g_\lambda(t,y)dy\nonumber\\
&=&{\E_x}[f(X(E_{\lambda}(t)))I(\tau_{D}(X)>E_{\lambda}(t))].\nonumber
\end{eqnarray}
The argument that
\[
{\E_{x}}[{f(X(E_{\lambda}{(t)}))I(
\tau_{D}(X)>
E_{\lambda}(t))}]={\E_{x}}[{f(X(E_{\lambda}{(t)}))I(
\tau_{D}(X(E_{\lambda}))> t)}]
\]
is similar to \cite[Corollary 3.2]{m-n-v-aop}.  Uniqueness follows by considering two solutions $u_1,u_2$ with the same initial data, and showing that $u_1-u_2\equiv 0$.
\end{proof}

\begin{remark}
In the special case where $L_D=\Delta$, the Laplacian operator, sufficient conditions for existence of strong solutions to \eqref{frac-derivative-bounded-dd} can be obtained from \cite[Corollary 3.4]{m-n-v-aop}.  Let $f\in C^{2k}_c(D)$ be a $2k$-times continuously differentiable function of compact support in D.  If $k>1+3d/4$, then
\eqref{frac-derivative-bounded-dd} has a classical (strong) solution. In particular, if $f\in C^{\infty}_c(D)$, then the
solution of \eqref{frac-derivative-bounded-dd} is in $C^\infty(D).$
\end{remark}

\begin{remark}
In the special case where $L_D=\Delta$ on an interval $(0,M)\subset \rr$, eigenfunctions and eigenvalues
are explicitly known, and solutions to the tempered fractional Cauchy problem can be made explicit.  Eigenvalues of the Laplacian on $(0,M)$ are
$(n\pi /M)^2$ for $n=1,2,\cdots  $ and the corresponding eigenfunctions are {$\frac2M\sin(n\pi x/M)$}.
Using this eigenfunction expansion,
the solution reads
$$
u(t,x)=
\sum_{n=1}^{\infty}\bar{f}(n)\psi_n(x)\check{g_\lambda}(t,\mu_n)=\sum_{n=1}^{\infty}\bar{f}(n)\sin(n\pi
x/M)\check{g_\lambda}(t,(n\pi /M)^2).
$$
\end{remark}

\end{document}